\documentclass[12pt, reqno]{amsart}
\usepackage{latexsym}
\usepackage{mathtools}
\usepackage{amsmath}
\usepackage{amsthm}
\usepackage{mathrsfs}
\usepackage{lmodern}
\usepackage[T1]{fontenc}
\usepackage{textcomp}
\usepackage[utf8]{inputenc}
\usepackage{color}
\usepackage{amssymb}
\usepackage{enumerate}
\usepackage[abbrev]{amsrefs}
\usepackage{indentfirst}
\usepackage{graphicx}
\usepackage{bmpsize}
\usepackage{hyperref}
\usepackage{enumitem}
\usepackage{enumerate}
\usepackage{url}
\usepackage{autobreak}
\usepackage[justification=justified]{caption}
\usepackage{bm}
\usepackage{here}
\usepackage[labelformat=empty,subrefformat=parens]{subcaption}
\renewcommand{\thefootnote}{} %footnote counter

\theoremstyle{plain} %text of this environment is typesetted in italics
\newtheorem{theorem}{\indent\sc Theorem}[section]
\newtheorem{lemma}[theorem]{\indent\sc Lemma}
\newtheorem{corollary}[theorem]{\indent\sc Corollary}
\newtheorem{proposition}[theorem]{\indent\sc Proposition}

\theoremstyle{definition} %text of this environment is typesetted in roman letters
\newtheorem{definition}[theorem]{\indent\sc Definition}
\newtheorem{remark}[theorem]{\indent\sc Remark}
\newtheorem{example}[theorem]{\indent\sc Example}

\newtheorem{question}[theorem]{\indent\sc Question}
\makeatletter
  
  \@addtoreset{equation}{section}
\makeatother
\newcommand{\cF}{\mathcal{F}}
\newcommand{\cL}{\mathcal{L}}
\newcommand{\cT}{\mathcal{T}}
\newcommand{\cB}{\mathcal{B}}
\newcommand{\bR}{\mathbb{R}}

\newcommand{\bZ}{\mathbb{Z}}
\newcommand{\qz}{\underbrace{0\cdots0}_{q}}
\newcommand{\Stab}{\operatorname{Stab}}
\newcommand{\oF}{\overrightarrow{F}}
\newcommand\Z[1]{\mathbb{Z} / #1\mathbb{Z}}

\newcommand\cinput[2]{\lower#1pt\hbox{\input{#2}}}

\begin{document}

\title{The $p$-colorable subgroup of Thompson's group}
\author{Yuya Kodama and Akihiro Takano}
\date{}
\renewcommand{\thefootnote}{\arabic{footnote}}  %number
\setcounter{footnote}{0} %footnote counter
\thanks{The first author was supported by JST, the establishment of university
fellowships towards the creation of science technology innovation, Grant Number JPMJFS2139.}

\begin{abstract}
Recently, Jones introduced a method of constructing knots and links from elements of Thompson's group $F$ by using its unitary representations.
He also defined several subgroups of $F$ as the stabilizer subgroups and some researchers studied them algebraically.
One of the subgroups is called the 3-colorable subgroup $\mathcal{F}$, and the authors proved that all knots and links obtained from non-trivial elements of $\mathcal{F}$ are 3-colorable.
In this paper, for any odd integer $p$ greater than two, we define the $p$-colorable subgroup of $F$ whose non-trivial elements yield $p$-colorable knots and links and show it is isomorphic to the certain Brown--Thompson group.
\end{abstract}

\maketitle

\section{Introduction}
%%%%%%%%%%%%%%%%%%%%%%%%%%%%%%%%%%%%%%%%%
Thompson's groups were defined by Richard Thompson in 1965. 
These groups were originally used for the viewpoint of logic. 
Nowadays, they are studied in various areas including geometric group theory. 
Especially for the group $F$, despite its uncomplicated group presentation, it is a mysterious one with many open problems such as amenability. 
Also, subgroups of $F$ are still actively studied in order to understand more about ``what $F$ is'' (see for instance \cite{MR4349287, MR4319969, MR4303330, MR4241460, aiello2021maximal, MR3710646, MR3565428}). 

Recently, Jones \cite{jones2017thompson} constructed unitary representations of Thompson's groups motivated by the creation of algebraic quantum field theories on the circle. 
In this process, he introduced a method of constructing knots and links from $F$ and proved that any of them can be obtained from a certain element of $F$ \cite{jones2017thompson, jones2019thompson}. 
In keeping with Alexander's theorem of the braid group, such a theorem is called Alexander's theorem. 

In his research program, he also found two interesting subgroups of $F$. 
One of them is the group called oriented subgroup $\oF$, and knots and links obtained from $\oF$ are naturally oriented \cite{jones2017thompson}. 
Jones \cite{jones2017thompson} proved a slightly weaker version of Alexander's theorem of $\oF$ for oriented knots and links, and Aiello \cite{aiello2020alexander} proved it completely. 
Golan and Sapir \cite{golan2017jones} characterized elements of $\oF$ by the perspective of a diagram group isomorphic to $F$ and constructed a natural generalization of $\oF$. 
They also characterized $\oF$ as a stabilizer of a subset of $\mathbb{Z}[1/2]\cap (0, 1)$. 
Another group is called $3$-colorable subgroup $\cF$ \cite{jones2018nogo}. 
Aiello and Naginibeda \cite{aiello2021maximal} characterized $\cF$ as a stabilizer of a subset of $\mathbb{Z}[1/2]\cap (0, 1)$. 
The authors \cite{kodama20223} proved that all knots and links obtained from non-trivial elements of $\cF$ are $3$-colorable. 
The $3$-colorability of knots and links is one of the well-known classical invariants in knot theory. 
We remark that Alexander's theorem of $\cF$ for $3$-colorable knots and links is still open. 

One of the major open problems in ``Thompson knot theory'' is Markov's theorem, that is, to describe completely what elements of $F$ give the same knot or link. 
For this purpose, it is important to compute more invariants.  
Therefore, in this paper, we focus on the $p$-colorability, which is a generalization of the $3$-colorability. 
We will define and investigate a new group called the $p$-colorable subgroup $\cF_p$. 

This paper is organized as follows: 
in Section \ref{section_preliminary}, we first review the definitions of the three groups called Thompson's group $F$, the Brown--Thompson group $F(n)$, and the $3$-colorable subgroup $\cF$. 
We then summarize Jones' construction from $F$ and the definition of the $p$-colorability, which is an invariant of knots and links. 
In Section \ref{section_definition_cFp}, we first define the group called the $p$-colorable subgroup $\cF_p$. 
The group $\cF_3$ coincides with the $3$-colorable subgroup $\cF$. 
We then show that any knot or link obtained from a non-trivial element of $\cF_p$ is $p$-colorable. 
In Section \ref{section_property_cFp}, we study some algebraic properties of the group $\cF_p$. 
%%%%%%%%%%%%%%%%%%%%%%%%%%%%%%%%%%%%%%%%%
\section{Preliminaries} \label{section_preliminary}
\subsection{The Brown--Thompson group $F(n)$} \label{Brown--Thompson}
Let $n \geq 2$. 
We define the Brown--Thompson group $F(n)$ by using tree diagrams. 
See \cite{burillo2001metrics} for details. 

We first define an \textbf{$n$-caret} as a graph with $n+1$ vertices and $n$ edges, 
where one vertex (\textbf{root}) has degree $n$ and the remaining vertices have degree one. 
By attaching the root of another $n$-caret to a vertex with degree one of an $n$-caret, we obtain a new graph. 
This graph is called an \textbf{$n$-ary tree with size two}. 
In general, attaching an $n$-caret to a vertex with degree one is called an \textbf{attachment}, and a graph obtained by iterating attachments $k$ times is called an \textbf{$n$-ary tree with size $k$}. 
The $n$-ary tree of size one and an $n$-caret is the same. 
We call a vertex with degree $n$ and a vertex with degree one as a \textbf{root} and a \textbf{leaf} of an $n$-ary tree, respectively. 
Note that the number of leaves is $k(n-1)+1$. 

Let $W_n$ be the set of all finite words on $\{0, \dots, n-1\}$. 
For each $n$-ary tree $T$, there exists an injection from the set of leaves of $T$ to $W_n$ defined as follows: 
since $T$ is a tree, for each leaf, there exists a unique path from the root to the leaf. 
By labeling each edge of $n$-carets with $0, \dots, n-1$ from the left, we obtain a finite word from each path. 
We frequently identify leaves with words, and define an order $<$ of leaves by the lexicographic order of words in $W_n$. 
For an $n$-ary word, we denote its length by $\|\cdot\|$. 

Let $\cT_n$ be the set of all pairs of $n$-ary trees with the same size. 
We define an equivalence relation on $\cT_n$ as follows: 
let $(T_+, T_-)$ be in $\cT_n$. 
Since the set of leaves of $T_+$ (resp.~$T_-$) is ordered, let $i_+$ (resp.~$i_-$) be the $i$-th leaf of $T_+$ (resp.~$T_-$). 
Then we can obtain an $n$-ary tree $T_+^\prime$ (resp.~$T_-^\prime$) by an attachment to $i_+$ (resp.~$i_-$). 
We call this way as an \textbf{insertion of $n$-carets} in $(T_+, T_-)$, and its inverse as a \textbf{reduction of $n$-carets}. 
Then we define an equivalence relation $\sim$ as the one generated by all insertions and reductions. 
We say that $(T_+, T_-) \in \cT_n$ is \textbf{reduced} if we can not reduce any $n$-carets. 
For each equivalence class, there exists a unique reduced representative  (see \cite{cannon1996intro} for $n=2$). 

The \textbf{Brown--Thompson group $F(n)$} is the group $(\cT_n/{\sim}, \times)$ where $\times$ is the following product: 
let $A=(A_+, A_-)$ and $B=(B_+, B_-)$ be in $\cT_n$. 
By insertions, there exist two elements $A^\prime=(A_+^\prime, A_-^\prime)$ and $B^\prime=(B_+^\prime, B_-^\prime)$ such that $A \sim A^\prime$, $B \sim B^\prime$, and $A_-^\prime=B_+^\prime$ hold. 
Then for the equivalence classes of $A$ and $B$, their product is an equivalence class of $(A_+^\prime, B_-^\prime)$. 
The identity element of $F(n)$ is an equivalence class of $(T, T)$, where $T$ is an $n$-ary tree. 
For an equivalence class of $(T_+, T_-) \in \cT_n$, its inverse element is the equivalence class of $(T_-, T_+)$. 

If $n=2$, then the group $F(2)$ is also called \textbf{Thompson's group $F$}, and a $2$-caret is also called a \textbf{caret}. 

There exist two well-known presentations for $F(n)$ \cite{brin1998automorphisms, guba1997diagram}: 
\begin{align*}
F(n) &\cong \langle x_0, x_1, x_2, \dots  \mid \text{$x_i^{-1} x_j x_i=x_{j+n-1}$ ($i<j$)}  \rangle \\
&\cong \left\langle x_0, x_1, \dots, x_{n-1} \;\middle|\;
\begin{array}{l} \text{${x_k}^{x_0}={x_k}^{x_i}$ ($1 \leq i<k\leq n-1$)}, \\
\text{${x_k}^{x_0x_0}={x_k}^{x_0x_i}$ ($1 \leq i, k\leq n-1$ and $k-1\leq i$)}, \\
{x_1}^{x_0x_0x_0}={x_1}^{x_0 x_0 x_{n-1}}
\end{array}
\right\rangle, 
\end{align*}
where $x^y$ denotes $y^{-1}x y$. 
Figure \ref{generator_F(n)} is the list of the set $\{x_0, x_1, \dots\}$. 
\begin{figure}[tbp]
\begin{center}
\includegraphics[height=150pt]{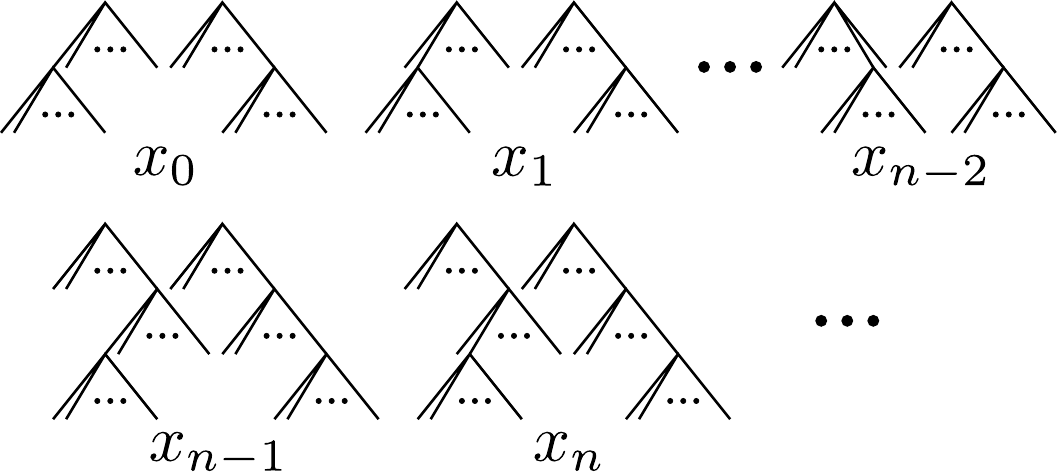}
\end{center}
\caption{The generating set of $F(n)$}
\label{generator_F(n)}
\end{figure}

The Brown--Thompson group $F(n)$ is also defined as a group of homeomorphisms on the closed interval $[0, 1]$. 
Since we only use this definition for the case $F$, we give the correspondence in that case. 
Define
\begin{align*}
f_0(x)&= \left \{
\begin{array}{cc}
2x & \mbox{\rm{if} $0 \leq x \leq \frac{1}{4}$} \\
x+\frac{1}{4} & \mbox{\rm{if} $\frac{1}{4} \leq x \leq \frac{1}{2}$} \\
\frac{x+1}{2} & \mbox{\rm{if} $\frac{1}{2} \leq x \leq 1$},
\end{array}
\right. 
&
f_1(x) &= \left \{
\begin{array}{cc}
x & \mbox{\rm{if} $0 \leq x \leq \frac{1}{2}$} \\
2x-\frac{1}{2} & \mbox{\rm{if} $\frac{1}{2} \leq x \leq \frac{5}{8}$} \\
x+\frac{1}{8} & \mbox{\rm{if} $\frac{5}{8} \leq x \leq \frac{3}{4}$} \\
\frac{x+1}{2} & \mbox{\rm{if} $\frac{3}{4} \leq x \leq 1$}. 
\end{array}
\right. 
\end{align*}
Then the group generated by $f_0$ and $f_1$ is isomorphic to $F$. 
The isomorphism is given by $x_0 \mapsto f_0, x_1 \mapsto f_1$. 
We remark that each element of $F$ gives two divisions of the closed interval $[0, 1]$. 
Each leaf $a_1 \cdots a_m$ corresponds to the leftmost element of the subinterval by the map $a_1 \cdots a_m \mapsto \Sigma_{i=1}^m {{a_i}/{2^i}}$. 
We write this map as $\rho$. 
See \cite{cannon1996intro} for details. 

In this paper, we mainly deal with $F$ and the Brown--Thompson groups in the case of $n=2^q$. 
We define embeddings from $F(2^q)$ into $F$. 
Let $T_q$ be a binary tree corresponding to a set $\{a_1 \cdots a_q \mid a_i \in \{0, 1\}\} \subset W_2$. 
See Figure \ref{T_q} for example. 
We define $\phi_q$ to be a map that replaces all $2^q$-carets of each element by $T_q$. 
See also \cite[example (1) in Section 3]{burillo2001metrics}. 
\begin{figure}[tbp]
\begin{center}
\includegraphics[height=70pt]{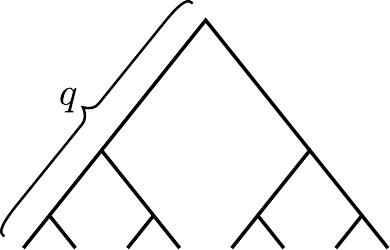}
\end{center}
\caption{A binary tree $T_q$ ($q=3$).}
\label{T_q}
\end{figure}

%%%%%%%%%%%%%%%%%%%%%%%%%%%%%%%%%%%%%%%%%
\subsection{The 3-colorable subgroup $\cF$} \label{3-colsub}
In this section, we summarize the definition of the $3$-colorable subgroup $\cF$ of Thompson's group $F$. 
See \cite{aiello2022thompson, kodama20223} for details. 

Let $(T_+, T_-)$ be in $\cT_2$. 
We assume that $T_+$ is descending tree with the root on the top, and $T_-$ is ascending tree with the root on the bottom. 
Since $T_+$ and $T_-$ have the same size, we attach their leaves in lexicographic order. 
The obtained graph is called \textbf{tree diagram}. 
See Figure \ref{tree_diagram} for example. 
\begin{figure}[tbp]
\begin{center}
\includegraphics[height=70pt]{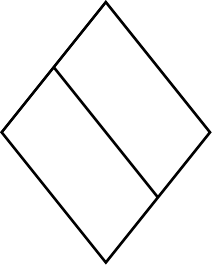}
\end{center}
\caption{A tree diagram of the generator $x_0$ in $F$. }
\label{tree_diagram}
\end{figure}

We put a tree diagram to $\mathbb{R}^2$ so that the root of $T_+$ is on $(0, 1) \in \mathbb{R}^2$, the root of $T_-$ is on $(0, -1) \in \mathbb{R}^2$, and the leaves are on $\mathbb{R}\times \{0\}$. 
Then we add two short edges, one from $(0, 1)$ to $(0, 2)$ and the other from $(0, -1)$ to $(0, -2)$. 
This gives a partition (called a \textbf{strip}) of $\mathbb{R} \times [-2, 2]$. 
A tree diagram is \textbf{$3$-strip-colorable} if we can assign a color $0$, $1$, or  $2$ to each partitioned region of the strip so that if two regions have a common edge, then they have different colors. 
By convention, if a tree diagram is $3$-strip-colorable, then we assign $0$ (resp.~$1$) to the infinite region in $(-\infty, 0] \times [-2, 2]$ (resp.~$[0, \infty) \times [-2, 2]$). 
This convention implies that the coloring of the region is unique. 
See Figure \ref{3_strip_coloring} for example. 
\begin{figure}[tbp]
\begin{center}
\includegraphics[height=120pt]{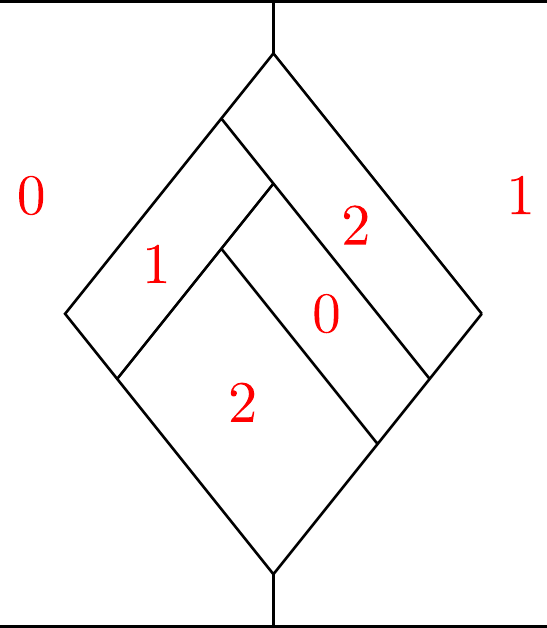}
\end{center}
\caption{A $3$-strip-colorable tree diagram}
\label{3_strip_coloring}
\end{figure}

We note that the $3$-strip-colorability is well-defined on $\cT_2/{\sim}=F$. 
Then we define the \textbf{$3$-colorable subgroup} as
\begin{align*}
\cF=\{(T_+, T_-) \in F \mid \text{$(T_+, T_-)$ is $3$-strip-colorable}\}. 
\end{align*}
The group $\cF$ and the Brown--Thompson group $F(4)$ have a close relationship. 
\begin{theorem}[{\cite[Theorem 1.3]{ren2018skein}}]
The group $\cF$ is isomorphic to the Brown--Thompson group $F(4)$. 
The isomorphism map is given by $\phi_2$. 
\end{theorem}

%%%%%%%%%%%%%%%%%%%%%%%%%%%%%%%%%%%%%%%%%
\subsection{Jones' construction} \label{jones}
In this section, we explain the method of constructing knots and links from elements of $F$ based on \cite{jones2019thompson}.
Let $(T_+, T_-)$ be a reduced tree diagram.

We construct the plane graph $\cB(T_+, T_-)$ containing a tree diagram $(T_+, T_-)$ as a subgraph as follows:
for any region, including the unbounded one, of $(T_+, T_-)$, the subgraph bounding it contains exactly two carets, one in $T_+$ and the other in $T_-$.
Then we connect roots of such carets by an edge in the region.
Figure~\ref{step1} is an example of this construction.

\begin{figure}[tbp]
\begin{center}
\includegraphics[height=135pt]{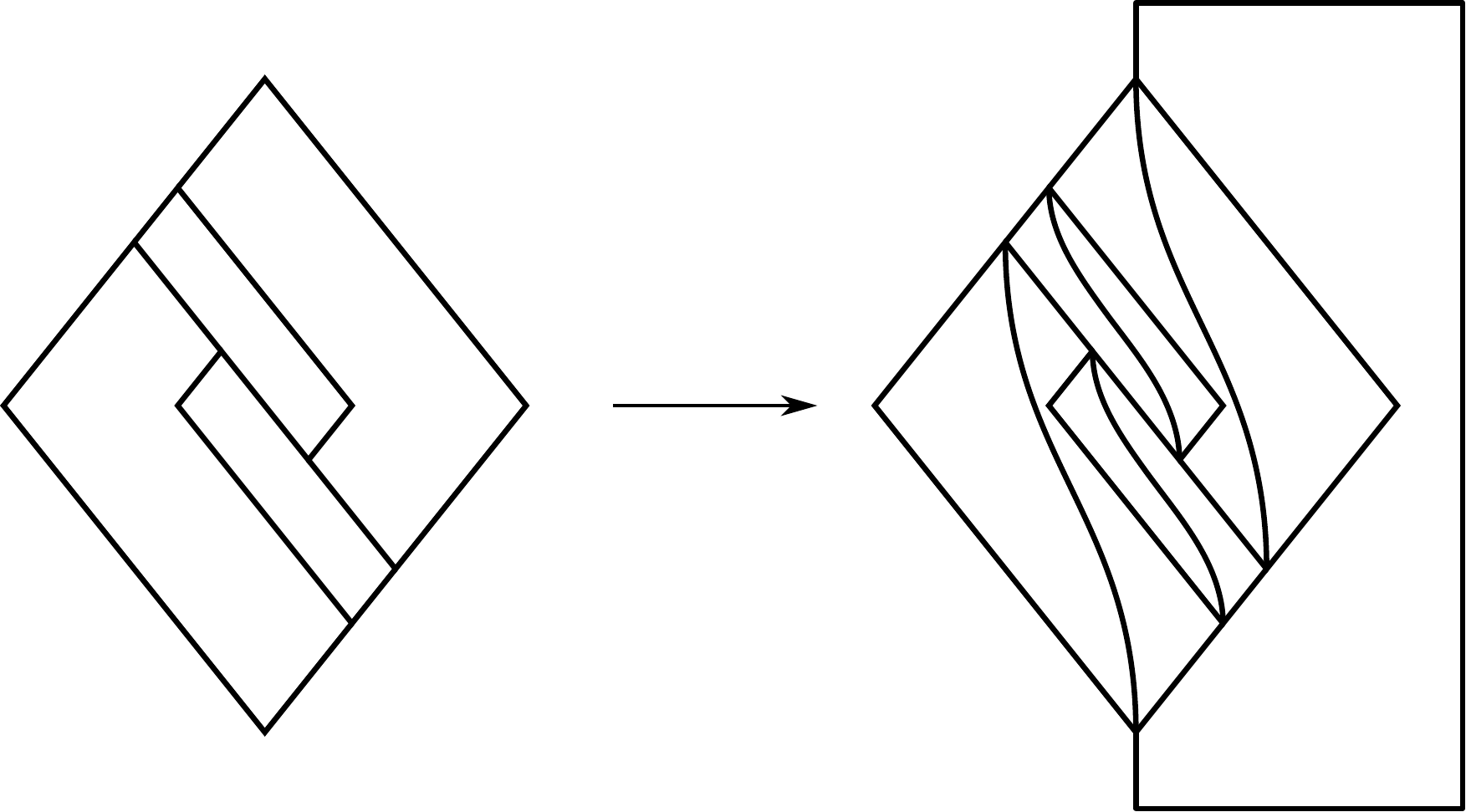}
\end{center}
\caption{The plane graph $\cB(T_+, T_-)$ associated with $(T_+, T_-)$.}
\label{step1}
\end{figure}

Since all vertices of the plane graph $\cB(T_+, T_-)$ are 4-valent, it can be regarded as a link projection.
Hence we obtain the link diagram $\cL(T_+, T_-)$ by turning each vertex into a crossing with the rules \cinput{5}{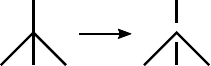} or \cinput{5}{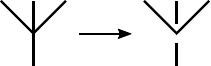}.
Figure~\ref{step2} is an example of this construction.

\begin{figure}[tbp]
\begin{center}
\includegraphics[height=135pt]{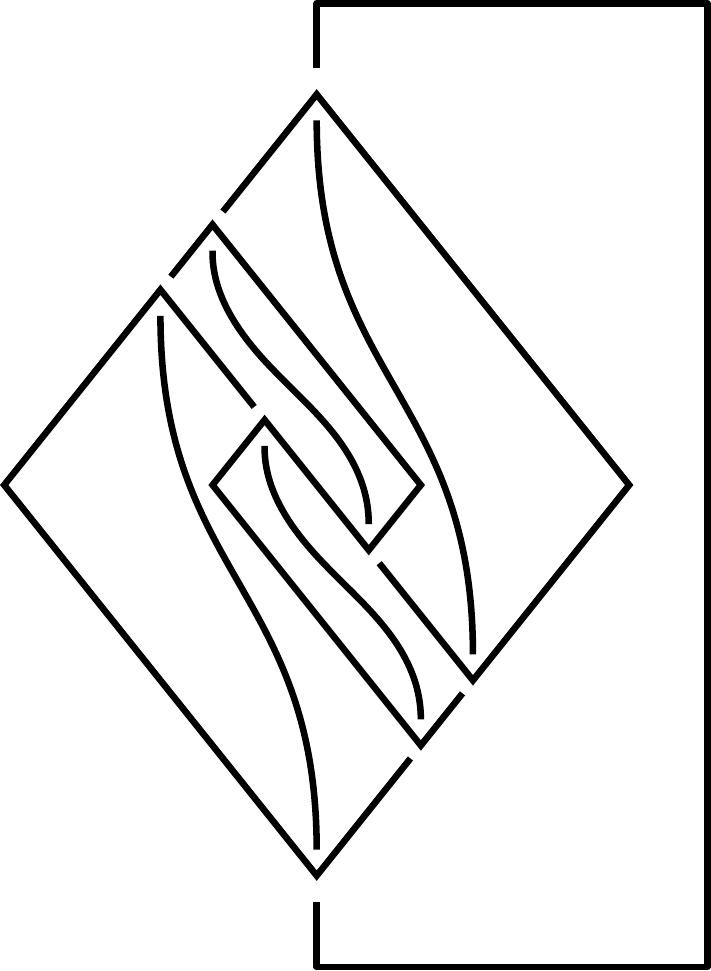}
\end{center}
\caption{The link diagram $\cL(T_+, T_-)$.}
\label{step2}
\end{figure}

\begin{remark} \label{mapL}
This construction can be applied for any non-reduced tree diagram.
However, the resulting link diagram is different from that of the equivalent reduced tree diagram:
let $(T'_+, T'_-)$ be a non-reduced tree diagram obtained by inserting carets into  a leaf of the reduced tree diagram $(T_+, T_-)$.
Then we have $\cL(T'_+, T'_-) = \cL(T_+, T_-) \sqcup \lower-1pt\hbox{$\bigcirc$}$; see Figure~\ref{caret_triv}.
On the other hand, each element $g \in F$ has the unique reduced tree diagram.
Therefore we can consider a map $\cL \colon F \to \{ $all knots and links$\}$ defined by $\cL(g)$ being the link diagram obtained from its reduced tree diagram.
Jones \cite[Theorem 5.3.1]{jones2017thompson} proved that this map is surjective, which is called Alexander's theorem.
On the other hand,  it is easy to see that this is not injective.
\end{remark}

\begin{figure}[tbp]
\begin{center}
\includegraphics[height=50pt]{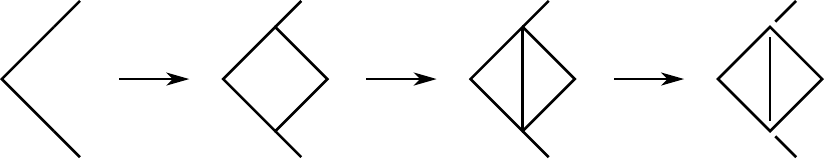}
\end{center}
\caption{A trivial link component by inserting carets.}
\label{caret_triv}
\end{figure}

\subsection{Dehn coloring}
In this section, we briefly review a link invariant obtained from a coloring of regions.
See \cite{carter2014coloring, kauffuman2006knot} for details.
Let $p \geq 3$ be an integer and $D$ be a diagram of a link $L$.

\begin{definition}
A \textbf{Dehn $p$-coloring} of $D$ is an assigning an element of $\Z{p}$ to each region of $D$ such that
\begin{itemize}
\item the unbounded region is assigned 0; and
\item it satisfies the following condition for any crossing:
\begin{align} \label{Dehn}
\cinput{25}{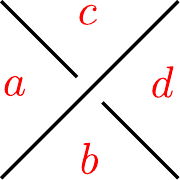} \hspace{30pt}
a + b \equiv c + d \pmod p.
\end{align}
\end{itemize}
\end{definition}

Recall that every link diagram has a unique checkerboard coloring such that the unbounded region is white.
A Dehn $p$-coloring is \textbf{trivial} if every white region is labeled by 0 and every black region is labeled by a fixed element $a \in \Z{p}$.
A link diagram is \textbf{Dehn $p$-colorable} if there exists a non-trivial Dehn $p$-coloring.
It is well known that the Dehn $p$-colorability is a link invariant.
Firure \ref{Dehn-example} is an example of a non-trivial Dehn 5-coloring of the figure-eight knot $4_1$.

There exists a classical invariant similar to the Dehn coloring, which is called the Fox coloring.
Carter, Silver, and Williams \cite{carter2014coloring} proved that these colorabilities are equivalent.

\begin{figure}[tbp]
\begin{center}
\includegraphics[height=90pt]{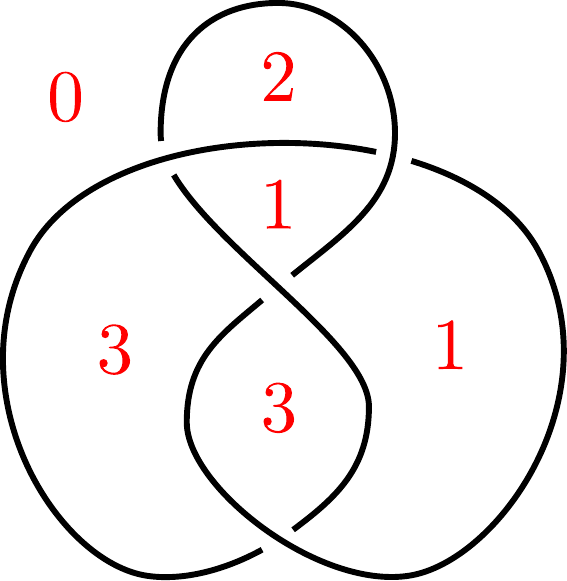}
\end{center}
\caption{A Dehn 5-coloring the figure-eight knot.}
\label{Dehn-example}
\end{figure}

\section{The $p$-colorable subgroup $\cF_p$}\label{section_definition_cFp}

\subsection{Definition}
In the rest of this paper, let $p \geq 3$ be an odd integer.
In order to define the $p$-colorable subgroup, we reinterpret the 3-colorable subgroup $\cF$ using the dyadic rationals.

Let $T$ be a binary tree.
Restricting the partition defined in Section \ref{3-colsub} to $\bR \times [0,2]$, we obtain a partition of $\bR \times [0,2]$ by $T$.
Moreover, there uniquely exists the assigning a number $0,1$, or $2$ to each partitioned region to satisfy the condition as in Section \ref{3-colsub}, where the left (resp.~right) unbounded region is assigned 0 (resp.~1).
For each leaf $i$ of $T$, let $\omega(i)$ be the number of the region to the left of $i$.
Then the following is obvious from the definition:

\begin{proposition}[{\cite[Proposition 2.1]{aiello2021maximal}}]
It holds
\begin{align*}
\cF = \left\{ (T_+, T_-) \in F \mid \forall i \geq 0, \omega(i_+) = \omega(i_-) \right\},
\end{align*}
where $i_+$ $($resp.~$i_-$$)$ is the leaf $i$ of $T_+$ $($resp.~$T_-$$)$.
\end{proposition}

Since we identified each leaf with a binary word in Section \ref{Brown--Thompson}, we obtain a map $\omega \colon W_2 \to \Z{3}$.
Aiello and Nagnibeda \cite{aiello2021maximal} proved the following:

\begin{proposition}[{\cite[Proposition 2.5]{aiello2021maximal}}]
For any binary word $a_1 \cdots a_n \in W_2$, it holds
\begin{align*}
\omega(a_1 \cdots a_n) \equiv \sum_{i=1}^{n} (-1)^i a_i \pmod 3.
\end{align*}
\end{proposition}

Recall the map $\rho \colon W_2 \to \bZ[1/2] \cap [0, 1)$ defined by $\rho(a_1 \cdots a_n) \coloneqq \sum_{i=1}^{n} a_i / 2^i$.
Since $-1 \equiv 1/2 \pmod 3$, we have
\begin{align*}
\omega(a_1 \cdots a_n) \equiv \rho(a_1 \cdots a_n) \pmod 3.
\end{align*}
Namely, each leaves $i_+$ and $i_-$ of any tree diagram $(T_+, T_-)$ of $\cF$ have the same dyadic rational modulo 3. 
Hence we are able to extend the 3-colorable subgroup $\cF$ to the case of an odd integer $p \geq 3$ by using the dyadic rationals.

\begin{definition} \label{Definition_cFp}
The \textbf{$p$-colorable subgroup} $\cF_p$ of Thompson's group $F$ is defined as the set
\begin{align*}
\cF_p \coloneqq \left\{ (T_+, T_-) \in F \mid \forall i \geq 0, \rho(i_+) \equiv \rho(i_-) \pmod p \right\}.
\end{align*}
\end{definition}

We will prove later that this set forms a subgroup of $F$.
Obviously, the 3-colorable subgroup $\cF$ coincides with $\cF_3$.

\begin{example} \label{ex3}
Consider the case $p=3$ and the following tree diagram:
\begin{align*}
(T_+, T_-) := \cinput{27}{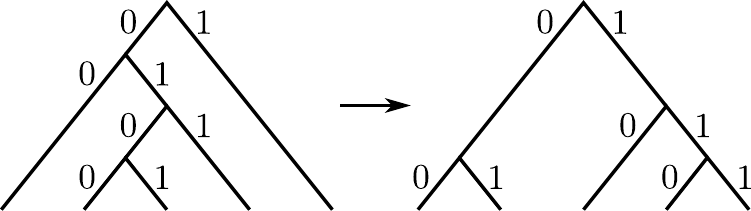}
\end{align*}
Then
\begin{align*}
&\rho(0_+) = \rho(00) = 0  =\rho(00) = \rho(0_-),\\
&\rho(1_+) = \rho(0100) = \frac{1}{4} \equiv 1 \equiv \frac{1}{4} = \rho(01) = \rho(1_-),\\
&\rho(2_+) = \rho(0101) = \frac{5}{16} \equiv 2 \equiv \frac{1}{2} = \rho(10) = \rho(2_-),\\
&\rho(3_+) = \rho(011) = \frac{3}{8} \equiv 0 \equiv \frac{3}{4} = \rho(110) = \rho(3_-),\ \textrm{and}\\
&\rho(4_+) = \rho(1) = \frac{1}{2} \equiv 2 \equiv \frac{7}{8} = \rho(111) = \rho(4_-).
\end{align*}
Therefore, this tree diagram is an element of $\cF_3$.
\end{example}

\begin{example} \label{ex7}
Consider the case $p=7$.
Then the following tree diagram is in $\cF_7$.
\begin{align*}
(T_+, T_-) := \cinput{35}{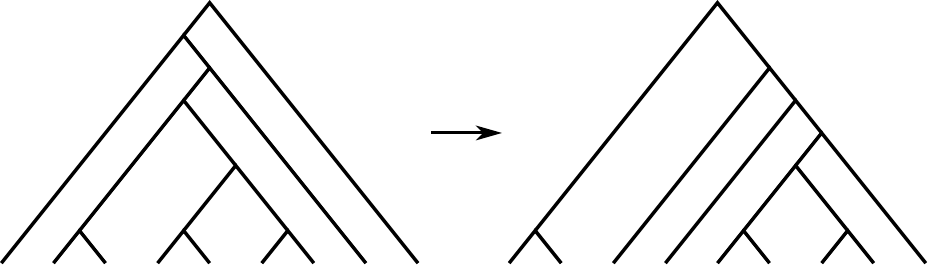}
\end{align*}
Indeed, we have
\begin{alignat*}{2}
&\rho(0_+) = 0 = \rho(0_-), & \quad &\rho(1_+) = \frac{1}{4} \equiv 2 \equiv \frac{1}{4} = \rho(1_-),\\
&\rho(2_+) = \frac{9}{32} \equiv 4 \equiv \frac{1}{2} = \rho(2_-), & \quad &\rho(3_+) = \frac{5}{16} \equiv 6 \equiv \frac{3}{4} = \rho(3_-),\\
&\rho(4_+) = \frac{21}{64} \equiv 0 \equiv \frac{7}{8} = \rho(4_-), & \quad &\rho(5_+) = \frac{11}{32} \equiv 1 \equiv \frac{57}{64} = \rho(5_-),\\
&\rho(6_+) = \frac{23}{64} \equiv 2 \equiv \frac{29}{32} = \rho(6_-), & \quad &\rho(7_+) = \frac{3}{8} \equiv 3 \equiv \frac{59}{64} = \rho(7_-),\ \textrm{and}\\
&\rho(8_+) = \frac{1}{2} \equiv 4 \equiv \frac{15}{16} = \rho(8_-).
\end{alignat*}
\end{example}

In the rest of this section, we rewrite the definition of $\cF_p$ in terms of the coloring of the regions of a tree diagram.

Let $T$ be a binary tree, and consider again a partition of $\bR^2 \times [0,2]$ by $T$.
Each region, except the right unbounded one, has a unique leaf such that the region is the left of the leaf.
Therefore we are able to assign each region the dyadic rational corresponding to its leaf.
We assume that the right unbounded region is assigned 1.

\begin{proposition} \label{dyadic_relation}
Let $a, b$ and $c$ be the dyadic rationals of the regions around a bifurcation below.
Then they satisfy $2b = a + c$.
\begin{align*}
\cinput{0}{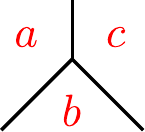}
\end{align*}
\end{proposition}

\begin{proof}
\textbf{Case 1.}
Suppose that two regions of $a$ and $c$ share an edge labeled by 0.
The dyadic rational $a$ corresponds to the leaf $w0\cdots 0$, where $w = a_1 \cdots a_n$ is a binary word with $a_n = 1$.
Then the binary word of the bifurcation above $b$ can be written as $w\underbrace{0 \cdots 0}_{k}$ for some $k \in \bZ$, and thus the dyadic rationals $b$ and $c$ correspond to $w\underbrace{0 \cdots 0}_{k}10 \cdots 0$ and $w\underbrace{0 \cdots 0}_{k-1}10 \cdots 0$, respectively.
This situation is described below:
\begin{align*}
\cinput{0}{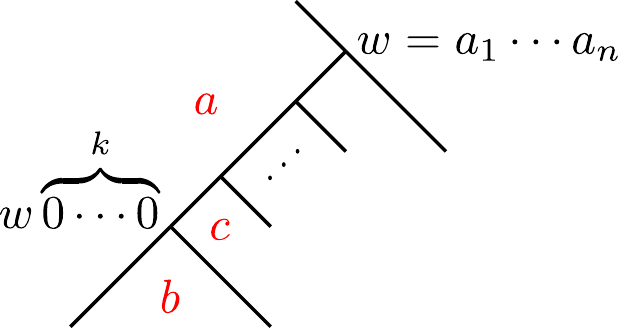}
\end{align*}
Hence
\begin{align*}
a &= \rho(w0 \cdots 0) = \rho(w) = \sum_{i=1}^{n} \frac{a_i}{2^i},\\
b &= \rho(w\underbrace{0 \cdots 0}_{k}10 \cdots 0) = \rho(w\underbrace{0 \cdots 0}_{k}1) = \sum_{i=1}^{n} \frac{a_i}{2^i} + \frac{1}{2^{n+k+1}},\ \textrm{and}\\
c &= \rho(w\underbrace{0 \cdots 0}_{k-1}10 \cdots 0) = \rho(w\underbrace{0 \cdots 0}_{k-1}1) = \sum_{i=1}^{n} \frac{a_i}{2^i} + \frac{1}{2^{n+k}}.
\end{align*}
Therefore
\begin{align*}
2b = 2 \sum_{i=1}^{n} \frac{a_i}{2^i} + \frac{1}{2^{n+k}} = \sum_{i=1}^{n} \frac{a_i}{2^i} + \left( \sum_{i=1}^{n} \frac{a_i}{2^i} + \frac{1}{2^{n+k}}\right) = a + c.
\end{align*}
If $w$ is the empty word, then $a = 0$ is in the left unbounded region:
\begin{align*}
\cinput{0}{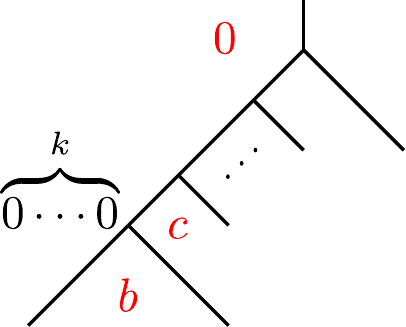}
\end{align*}
We have $b = 1/2^{k+1}$ and $c = 1/2^k$, and thus the equality is clear.

\noindent
\textbf{Case 2.}
Suppose that two regions of $a$ and $c$ share an edge labeled by 1.
If the dyadic rational $c$ is not in the right unbounded region, then $c$ corresponds to the leaf $w10 \cdots 0$, where $w = a_1 \cdots a_n \in W_2$.
Then the binary word of the bifurcation above $b$ can be written as $w0\underbrace{1 \cdots 1}_{k}$ for some $k \in \bZ$, and thus the dyadic rationals $a$ and $b$ correspond to $w0\underbrace{1 \cdots 1}_{k}0 \cdots 0$ and $w0\underbrace{1 \cdots 1}_{k}10 \cdots 0$, respectively.
This situation is described below:
\begin{align*}
\cinput{0}{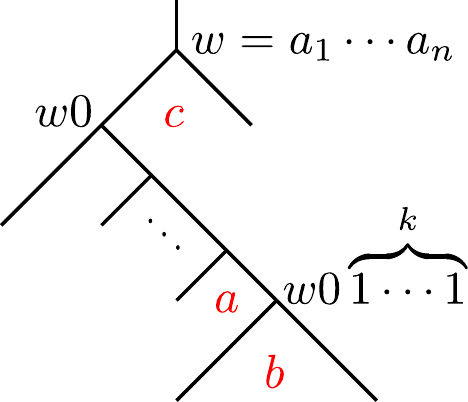}
\end{align*}
Hence
\begin{align*}
a &= \rho(w0\underbrace{1 \cdots 1}_{k}) = \sum_{i=1}^{n} \frac{a_i}{2^i} + \sum_{j=1}^{k} \frac{1}{2^{n+j+1}},\\
b &= \rho(w0\underbrace{1 \cdots 1}_{k}1) = \sum_{i=1}^{n} \frac{a_i}{2^i} + \sum_{j=1}^{k+1} \frac{1}{2^{n+j+1}},\ \textrm{and}\\
c &= \rho(w1) = \sum_{i=1}^{n} \frac{a_i}{2^i} + \frac{1}{2^{n+1}}.
\end{align*}
Therefore
\begin{align*}
2b = 2 \sum_{i=1}^{n} \frac{a_i}{2^i} + \sum_{j=1}^{k+1} \frac{1}{2^{n+j}} = \left( \sum_{i=1}^{n} \frac{a_i}{2^i} + \sum_{j=1}^{k} \frac{1}{2^{n+j+1}} \right) + \left( \sum_{i=1}^{n} \frac{a_i}{2^i} + \frac{1}{2^{n+1}} \right) = a + c.
\end{align*}
Finally, consider the case of $c$ in the right unbounded region:
\begin{align*}
\cinput{0}{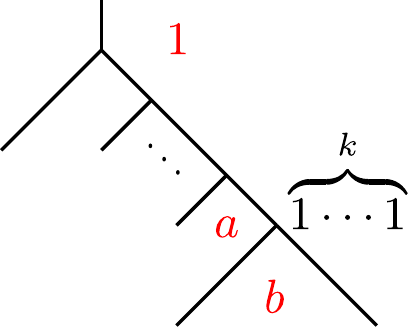}
\end{align*}
Then we have $c = 1, a = \sum_{i=1}^{k} 1/2^i$, and $b = \sum_{i=1}^{k+1} 1/2^i$, and thus the equality can be easily checked.
\end{proof}

A tree diagram $(T_+, T_-)$ is said to be \textbf{$p$-strip-colorable} if there exists a coloring with an element of $\Z{p}$ to each partitioned region such that
\begin{itemize}
\item the left and right unbounded regions are colored with 0 and 1, respectively; and
\item three colors around each bifurcation satisfy the following condition:
\begin{align} \label{tree_col}
\cinput{15}{bifurcation_relation.pdf_tex} \hspace{30pt} 2b \equiv a + c \pmod p.
\end{align}
\end{itemize}
From Proposition \ref{dyadic_relation}, each of binary trees $T_+$ and $T_-$ has a unique coloring of regions satisfying the conditions above induced from the dyadic rationals of the leaves.
Hence, the $p$-strip-colorability of a tree diagram is equivalent to the condition of $\cF_p$:

\begin{proposition}
It holds
\begin{align*}
\cF_p = \left\{ (T_+, T_-) \in F \mid (T_+, T_-)\ \textrm{is $p$-strip-colorable.}\right\}.
\end{align*}
\end{proposition}

In order to prove that this set forms a subgroup of $F$, it is enough to check that inserting carets preserves the $p$-strip-colorability, which is obvious from the figure below:
\begin{align*}
\cinput{35}{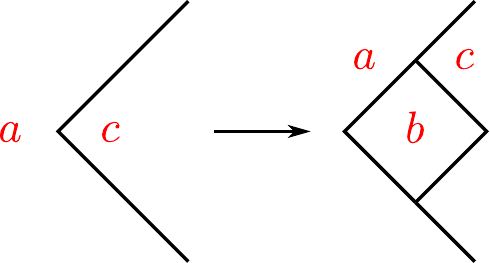} \hspace{30pt}
2b \equiv a+c \pmod p.
\end{align*}

\subsection{Knots and links obtained from $\cF_p$}
In this section, we describe a property of knots and links obtained from elements in the $p$-colorable subgroup $\cF_p$.

\begin{theorem} \label{p-col_knot}
All knots and links obtained from non-trivial elements of $\cF_p$ are $p$-colorable.
\end{theorem}

\begin{proof}
In order to show this theorem, we will find a non-trivial Dehn $p$-coloring of the link diagram $\cL(g)$ for any non-trivial element $g = (T_+, T_-) \in \cF_p$.

Let $(T_+, T_-)$ be a reduced tree diagram in $\cF_p$.
In Jones' construction in Section \ref{jones}, we obtain the plane graph $\cB(T_+, T_-)$ by dividing each region by an edge.
Namely, there are two regions, called the \textit{left region} and \textit{right region}, for each bifurcation of $T_+$ (or $T_-$).
Then we assign a number to the regions of $\cB(T_+, T_-)$ by using the $p$-strip-coloring like below:
\begin{align*}
\cinput{0}{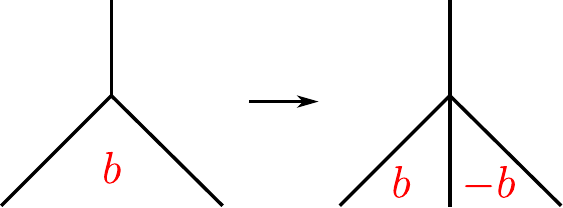}
\end{align*}
By the replacing rules of Jones' construction, we obtain the link diagram $\cL(T_+, T_-)$.
In fact, this satisfies the condition (\ref{Dehn}) for each crossing (i.e.~bifurcation).
Indeed, the upper left (resp.~right) region of a bifurcation in $\cB(T_+, T_-)$ is actually the \textit{right region} (resp.~\textit{left region}) of another bifurcation above.
This situation is described below: Case 1 in the proof of Proposition \ref{dyadic_relation} is of the form
\begin{align*}
\cinput{0}{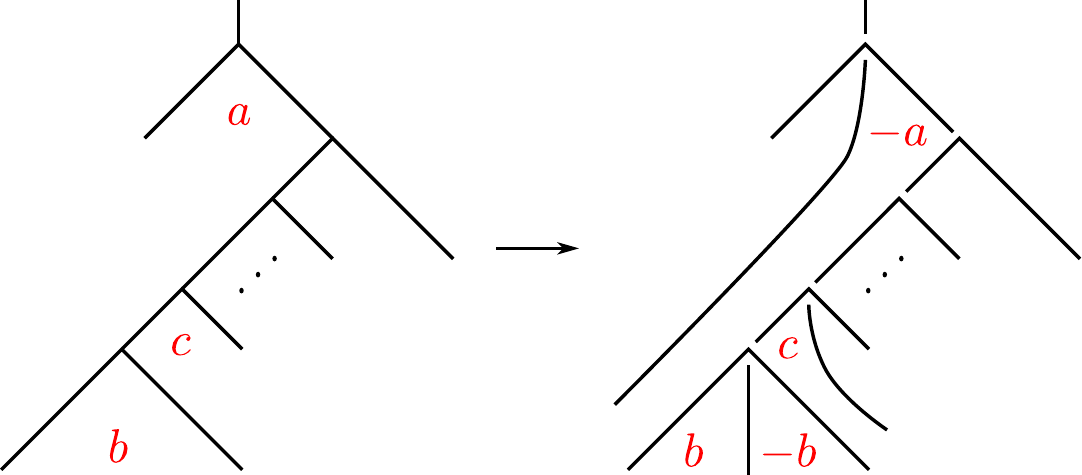}
\end{align*}
and Case 2 is of the form
\begin{align*}
\cinput{0}{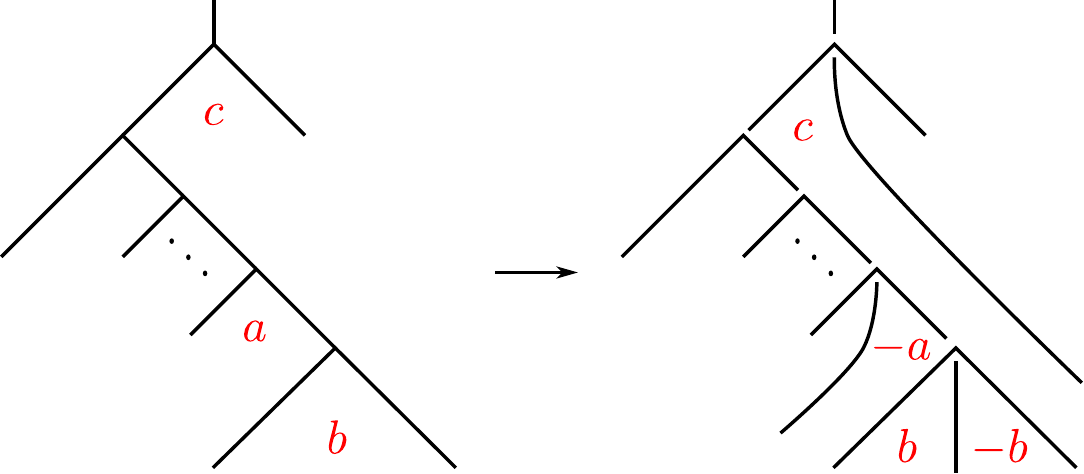}
\end{align*}
Therefore, the sum of numbers of left (resp.~right) two regions around a bifurcation is $-a+b$ (resp.~$-b+c$), and they are equal modulo $p$ from the condition (\ref{tree_col}).
We easily see that this is a non-trivial Dehn $p$-coloring of $\cL(T_+, T_-)$.
\end{proof}

\begin{example} \label{ex3_knot}
Consider the tree diagram $(T_+, T_-)$ in Example \ref{ex3}.
Then the knot $\cL(T_+, T_-)$ is the trefoil knot $3_1$ and it is 3-colorable; see Figure~\ref{3-col_example_knot}.
\begin{figure}[tbp]
\begin{center}
\includegraphics[height=145pt]{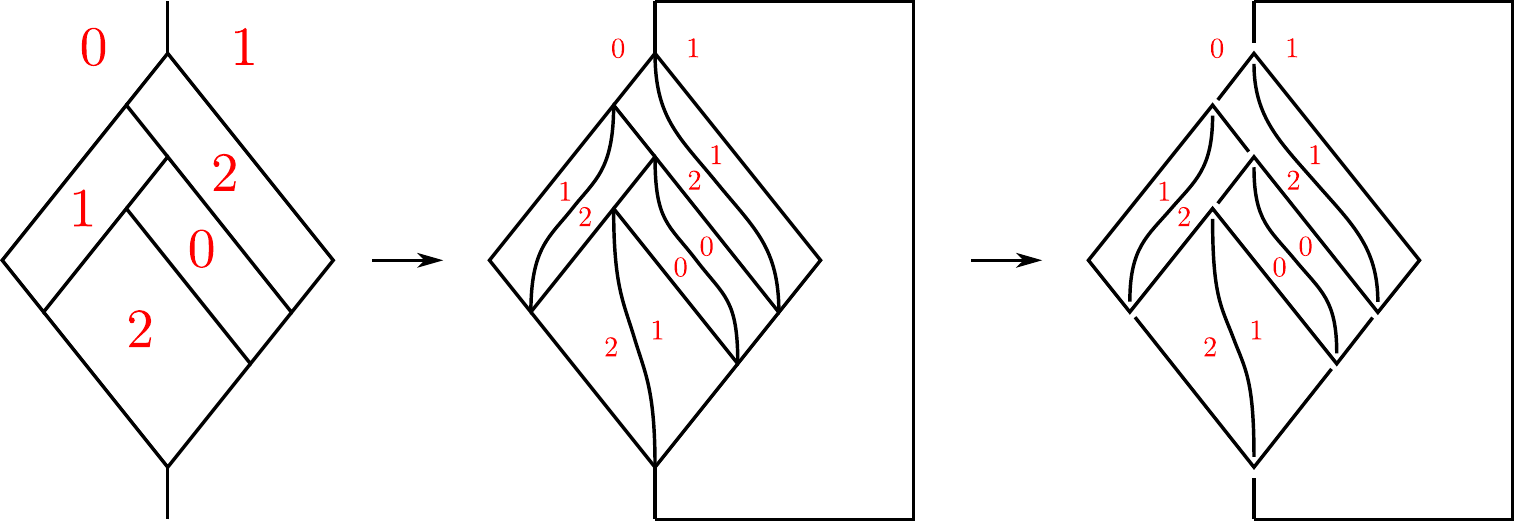}
\end{center}
\caption{The leftmost figure is the 3-strip-coloring of $(T_+, T_-)$ in Example \ref{ex3}, and the rightmost one is the associated non-trivial Dehn 3-coloring of the knot diagram $\cL(T_+, T_-)$.}
\label{3-col_example_knot}
\end{figure}
\end{example}

\begin{example} \label{ex7_knot}
Consider the tree diagram $(T_+, T_-)$ in Example \ref{ex7}.
Then the knot $\cL(T_+, T_-)$ is a twist knot $5_2$ which is $7$-colorable; see Figure~\ref{7-col_example_knot}.
\begin{figure}[tbp]
\begin{center}
\includegraphics[height=180pt]{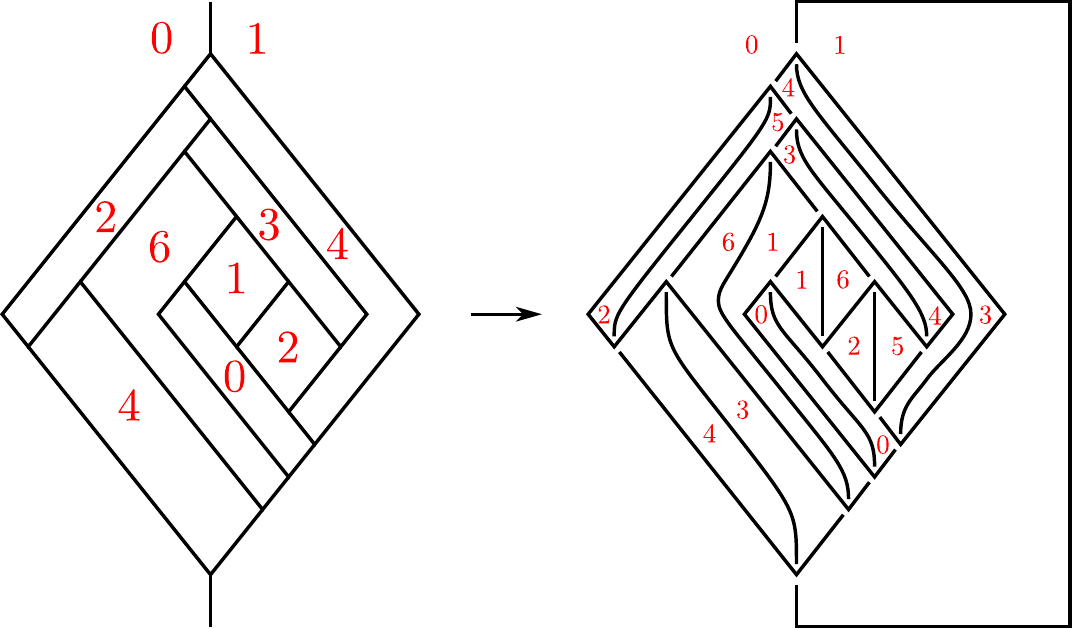}
\end{center}
\caption{The left figure is the 7-strip-coloring of $(T_+, T_-)$ in Example \ref{ex7}, and the right one is the associated non-trivial Dehn 7-coloring of the knot diagram $\cL(T_+, T_-)$.}
\label{7-col_example_knot}
\end{figure}
\end{example}

\begin{remark}
For an odd integer $p \geq 3$, let $q$ be a positive integer satisfying $2^q \equiv 1 \pmod p$.
Consider the tree diagram $(T_+, T_-)$ as follows:
\begin{align*}
(T_+, T_-) \coloneqq \cinput{20}{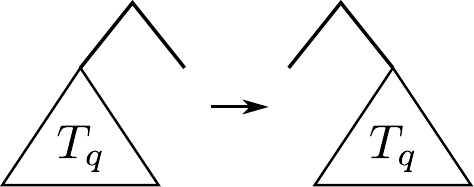}
\end{align*}
Then $(T_+, T_-)$ is in $\cF_p$ from Proposition \ref{cFp-length} below.
By direct calculation, we see that the number of components of $\cL(T_+, T_-)$ is one.
Hence, there always exists an element in $\cF_p$ which produces a non-trivial $p$-colorable knot.
For example, if $p=3$ and $q=2$, then we obtain the trefoil knot $3_1$.
Also if $p=7$ and $q=3$, then we obtain the knot $7_7$.
\end{remark}

%%%%%%%%%%%%%%%%%%%%%%%%%%%%%%%%%%%%%%%%%
\section{Prorerties of the $p$-colorable subgroup $\cF_p$} \label{section_property_cFp}
%%%%%%%%%%%%%%%%%%%%%%%%%%%%%%%%%%%%%%%%%
\subsection{A relationship with the Brown-Thompson group} \label{subsec_isomorphism}
In this section, we give another interpretation of the $p$-colorable subgroup $\cF_p$ and prove that it is isomorphic to the certain Brown--Thompson group.

Let $(T_+, T_-)$ be a tree diagram of $\cF_p$.
Fix a leaf $i_+$ of $T_+$ (resp.~$i_-$ of $T_-$), and suppose that $\| i_+ \| = k$ (resp.~$\| i_- \| = m$).
Inserting carets to a leaf $i$, we obtain an equivalent tree diagram $(T'_+, T'_-)$.
Then the lengths of $(i+1)_+$ and $(i+1)_-$ of $T'_+$ and $T'_-$ are $k+1$ and $m+1$, respectively.
Therefore they must satisfy $2^{k+1} \equiv 2^{m+1} \pmod p$.
This is equivalent to $k \equiv m \pmod q$, where $q$ is the smallest positive integer satisfying $2^q \equiv 1 \pmod p$.
Namely, each leaves $i_+$ and $i_-$ of a tree diagram of $\cF_p$ have the same length modulo $q$.
Moreover we have the following result:

\begin{proposition} \label{q_condition} \label{cFp-length}
It holds
\begin{align*}
\cF_p =  \left\{ (T_+, T_-) \in F \mid \forall i \geq 0, \| i_+ \| \equiv \| i_- \| \pmod q \right\}.
\end{align*}
\end{proposition}

In order to show the converse inclusion, we use a binary tree $T_q$ with $2^q$ leaves defined in Section \ref{Brown--Thompson}.
\begin{lemma} \label{q_length}
Let $T$ be a binary tree such that each leaf $i$ satisfies $\| i \| = k_i q$ for some positive integer $k_i$.
Then $T$ can be obtained by attaching some $T_q$'s.
\end{lemma}

\begin{proof}
Choose the leaf $l$ of $T$ such that it is the leftmost leaf with the maximum length.
Let $a_1 \cdots a_{k_lq}$ be the binary word corresponding to the leaf $l$.
Then we have $a_{(k_l-1)q+1} = \cdots = a_{k_lq} = 0$.
Indeed, if there exists a number $i$ such that $a_{(k_l-1)q+i} = 1$, then there exists a leaf $m < l$ of $T$ such that $k_m = k_l$, which is a contradiction to the leaf $l$ being the leftmost leaf with the maximum length.
Therefore the lengths of all of the leaves from $l$ to $l+2^q-1$ are $k_lq$.
Note that their binary words are of the form $a_1 \cdots a_{(k_l-1)q}b_1 \cdots b_q$, where $b_j \in \{ 0, 1 \}$.
Since the number of them is $2^q$, the set $\{ b_1 \cdots b_q \mid b_j \in \{ 0, 1 \} \}$ corresponds to the binary tree $T_q$.
Let $T'$ be the binary tree obtained from $T$ by removing this binary tree $T_q$.
Applying the same argument to $T'$, we find a binary tree $T_q$ again from it.
\end{proof}

\begin{lemma} \label{q_dyadic}
Suppose that a binary tree $T$ satisfies the assumption in Lemma $\ref{q_length}$.
Then the dyadic rationals of its leaves modulo $p$ are of the form
\begin{align*}
0, 1, 2, \ldots, p-1, 0, 1, 2, \ldots, p-1, \ldots, 0, 1, 2, \ldots, p-1, 0
\end{align*}
from left to right.
\end{lemma}

\begin{proof}
The dyadic rational of a leaf $i\ (0 \leq i \leq 2^q-1)$ of $T_q$ is $i / 2^q$.
Therefore, by the condition $2^q \equiv 1 \pmod p$, the dyadic rationals of leaves of $T_q$ modulo $p$ are of the form
\begin{align*}
0, 1, 2, \ldots, p-1, \ldots, 0, 1, 2, \ldots, p-1, 0
\end{align*}
from left to right.

Note that any binary tree satisfying the assumption in Lemma \ref{q_length} can be obtained by connecting some $T_q$'s.
We suppose that such a binary tree $T$ satisfies the result of Lemma \ref{q_dyadic}.
Let $T'$ be a binary tree obtained from $T$ by attaching the tree $T_q$ to a leaf $i$ of $T$.
If the dyadic rational of $i$ of $T$ modulo $p$ is $r$, then that of the leaf $i$ of $T'$ is the same.
As above, the dyadic rationals from $i$ to $i+2^q-1$ of $T'$ are of the form
\begin{align*}
r, r+1, r+2, \ldots, r-1, \ldots, r, r+1, r+2, \ldots, r-1, r.
\end{align*}
By the assumption, the dyadic rational of the leaf $i+2^q$ of $T'$, that is, the leaf $i+1$ of $T$ modulo $p$ is $r+1$.
Hence the binary tree $T'$ also satisfies the desired result.
\end{proof}

\begin{proof}[Proof of Proposition $\ref{q_condition}$]
Let $(T_+, T_-)$ be a tree diagram with $\| i_+ \| \equiv \| i_- \| \pmod q$ for any leaf.
By adding carets, we are able to obtain a tree diagram $(T'_+, T'_-)$ equivalent to $(T_+, T_-)$ satisfying the assumption in Lemma $\ref{q_length}$.
From Lemma \ref{q_dyadic}, the dyadic rationals of leaves of $T'_+$ and $T'_-$ modulo $p$ are both of the form
\begin{align*}
0, 1, 2, \ldots, p-1, \ldots, 0, 1, 2, \ldots, p-1, 0
\end{align*}
from left to right.
Therefore $(T'_+, T'_-) \sim (T_+, T_-)$ is in $\cF_p$.
\end{proof}

From Lemma \ref{q_length}, each element of $\cF_p$ has a tree diagram obtained by attaching some $T_q$'s.
Recalling the definition of the embedding $\phi_q \colon F(2^q) \to F$, we see that replacing all $2^q$-carets of an element of $F(2^q)$ by $T_q$ gives an element of $\cF_p$.
Therefore, we obtain the following theorem:

\begin{theorem}\label{Theorem_BT_cFp}
The $p$-colorable subgroup $\cF_p$ is isomorphic to the Brown--Thompson group $F(2^q)$.
The isomorphism map is given by $\phi_q$. 
\end{theorem}

In particular, for any integer $n \geq 2$, the Brown--Thompson group $F(2^n)$ is isomorphic to the $(2^n-1)$-colorable subgroup $\cF_{2^n-1}$.

\begin{corollary} \label{2n-ary}
Let $n \geq 2$ be an integer.
Then all knots and links obtained from non-trivial elements of the subgroup $\phi_n(F(2^n))$ of $F$ are $(2^n-1)$-colorable.
\end{corollary}

\begin{example}
The 3-colorable subgroup $\cF_3$ is isomorphic to $F(4)$, which is already shown by Ren \cite[Theorem 1.3]{ren2018skein}.
Also the subgroups $\cF_5$ and $\cF_7$ are isomorphic to $F(16)$ and $F(8)$, respectively.
Moreover, the 15-colorable subgroup $\cF_{15}$ is also isomorphic to $F(16)$, and thus non-trivial elements of $\cF_5 (= \cF_{15})$ produce 15-colorable links.
\end{example}

From Corollary \ref{2n-ary}, there exists a natural question about the properties of knots and links and the $p$-colorable subgroup $\cF_p$, which is a generalization of \cite[Question 3.5]{kodama20223}:

\begin{question}
Is the map
\begin{align*}
\cL \colon \cF_{2^n-1} (= \phi_n (F(2^n))) \setminus \{ 1 \} \to \left\{ \textrm{all $(2^n-1)$-colorable knots and links} \right\}
\end{align*}defined in Remark \ref{mapL} surjective?
\end{question}

%%%%%%%%%%%%%%%%%%%%%%%%%%%%%%%%%%%%%%%%%
\subsection{The $p$-colorable subgroup as a stabilizer subgroup}
%%%%%%%%%%%%%%%%%%%%%%%%%%%%%%%%%%%%%%%%%
In this section, we generalize the argument in \cite[Section 2]{aiello2021maximal} for $\cF_p$. 
Although some arguments are slightly modified, it yields the same result in \cite{aiello2021maximal} when $p=3$. 

Let $q$ be a positive integer defined in Section \ref{subsec_isomorphism}. 
The set $W_2 \setminus \{0^i \mid i \geq 0\}$ is denoted by $W_2^-$. 
We first define an equivalence relation $\approx$ on $W_2^-$ as the one generated by $v \approx v \qz$ where $v$ is in $W_2^-$. 
Set $W^q \coloneqq W_2^- / {\approx}$. 
For $u \in W_2^-$, we write its equivalence class as $[u]$. 
We define a right group action of $F$ on $W^q$ by rewriting subwords by two binary words corresponding to leaves. 
We note that this group action induces the natural action of $F$ on $[0, 1]$. 
\begin{example}
The action of $x_0^{-1} \in F$ on $[01]$ and $[1]=[1\qz]$ are $[001]$ and $[01 \underbrace{0\cdots0}_{q-1}]$, respectively. 
\end{example}

Recall that for $u \in W_2$, $\|u\|$ denotes its length as a binary word. 
Also, recall the map $\rho$, which is defined by $u=a_1 \cdots a_m \mapsto  \Sigma_{i=1}^m {{a_i}/{2^i}}$. 
We consider the following sets: 
\begin{align*}
S_i^q \coloneqq \{[u] \in W^q \mid \|[u]\| \in q\mathbb{N}, \rho([u]) \equiv i \pmod p \}
\end{align*}
where $i \in \mathbb{Z}/ p\mathbb{Z}$.
Note that the above two conditions are well-defined. 
\begin{remark}
The map 
\begin{align*}
\rho: \coprod_i S_i^q \to \mathbb{Z}[1/2] \cap (0, 1)
\end{align*}
is well-defined and bijective. 
\end{remark}
We start with an investigation of the stabilizers. 
\begin{lemma}
For each $i \in \mathbb{Z}/p\mathbb{Z}$, 
\begin{align*}
\Stab(S_i^q) \subset \Stab(S_{i-1}^q)
\end{align*}
holds. 
\begin{proof}
Let $f$ be an element of $\Stab(S_i^q)$. 
We show that if $[u]$ is in $S_{i-1}^q$ then $[u]\cdot f$ is also in $S_{i-1}^q$. 
Let $u$ be a representative of $[u]$ such that $u \cdot f$ is defined and we write $u \cdot f$ as $a_1 \cdots a_m$. 
Since $\|u\| \in q\mathbb{N}$, we have 
\begin{align*}
\rho(u \underbrace{0\cdots0}_{q-1}1)\equiv (i-1)+1=i \pmod p. 
\end{align*}
This implies that 
\begin{align*}
(u \underbrace{0\cdots0}_{q-1}1)\cdot f=a_1 \cdots a_m \underbrace{0\cdots0}_{q-1}1
\end{align*}
is in $S_i^q$. 
Since $\|a_1 \cdots a_m \underbrace{0\cdots0}_{q-1}1\|$ is in $q\mathbb{N}$, $\|a_1 \cdots a_m\|$ is also in $q\mathbb{N}$. 
Also, since 
\begin{align*}
\rho(a_1 \cdots a_m \underbrace{0\cdots0}_{q-1}1)\equiv i
\end{align*}
holds, we have $\rho(a_1 \cdots a_m) \equiv i-1$. 
\end{proof}
\end{lemma}
This lemma implies that $\Stab(S_i^q)=\Stab(S_j^q)$ holds for any $i, j \in \mathbb{Z}/p\mathbb{Z}$. 
We are now ready for the proof of the main theorem in this section. 
\begin{theorem}
The $p$-colorable subgroup $\cF_p$ coincides with $\Stab(S_i)$ for all $i \in \mathbb{Z}/p\mathbb{Z}$. 
\begin{proof}
For the inclusion $\cF_p \subset \Stab(S_i^q)$, we only need to check the generators (and their inverse) of $\cF_p$. 
By Theorem \ref{Theorem_BT_cFp}, $\cF_p$ is generated by $\phi_q(x_0), \dots,  \phi_q(x_{2^q-1})$, where each $x_k$ is in the Brown--Thompson group $F(2^q)$. 
Then it is clear from the definition of the binary tree $T_q$. 

Let $f=(T_+, T_-) \in \Stab(S_i)=\bigcup_{i}{\Stab(S_i)}$. 
According to Definition \ref{Definition_cFp}, we show that $\rho(l_+) \equiv \rho(l_-) \pmod p$ holds, where $l_+$ (resp.~$l_-$) is the binary word corresponding to the $l$-th leaf of $T_+$ (resp.~$T_-$). 
If $l_+=\underbrace{0\cdots0}_{j}$ for some $j \geq 1$, then it is clear that $\rho(l_+)=\rho(l_-)=0$ holds. 
If $l_+$ is otherwise, we consider the binary word $u\coloneqq l_+\underbrace{0\cdots0}_{j}$ with some $j\geq 0$ such that $[u]$ is in $W^q$. 
Then since $f \in \bigcup_{i}{\Stab(S_i)}$, we have
\begin{align*}
\rho(l_+)=\rho(l_+\underbrace{0\cdots0}_{j}) \equiv \rho(l_- \underbrace{0\cdots0}_{j})=\rho(l_-). 
\end{align*}
This completes the proof. 
\end{proof}
\end{theorem}
\begin{remark}
Due to the work of Golan \cite[Corollary 5.7]{golan2016generation}, the $p$-colorable subgroup $\cF_p$ is closed. 
\end{remark}
%%%%%%%%%%%%%%%%%%%%%%%%%%%%%%%%%%%%%%%%%
\section*{Acknowledgements}
We are grateful to Professor Motoko Kato for her suggestion. 
We would like to thank Professor Tomohiro Fukaya who is the first author's supervisor for his helpful comments. 
We also wish to thank Professor Takuya Sakasai who is the second author's supervisor for his helpful comments.

%%%%%%%%%%%%%%%%%%%%%%%%%%%%%%%%%%%%%%%%%
%%%%%%%%%%%%%%%%%%%%%%%%%%%%%%%%%%%%%%%%%
\bibliographystyle{plain}
\bibliography{bib1} 
%%%%%%%%%%%%%%%%%%%%%%%%%%%%%%%%%%%%%%%%%
%%%%%%%%%%%%%%%%%%%%%%%%%%%%%%%%%%%%%%%%%
%%%%%%%%%%%%%%%%%%%%%%%%%%%%%%%%%%%%%%%%%
\bigskip
\address{
%Department of Mathematical Sciences,
%Tokyo Metropolitan University,
%Minami-Osawa Hachioji, Tokyo, 192-0397, Japan
DEPARTMENT OF MATHEMATICAL SCIENCES,
TOKYO METROPOLITAN UNIVERSITY,
MINAMI-OSAWA HACHIOJI, TOKYO, 192-0397, JAPAN
}

\textit{E-mail address}: \href{mailto:kodama-yuya@ed.tmu.ac.jp}{\texttt{kodama-yuya@ed.tmu.ac.jp}}

\address{GRADUATE SCHOOL OF MATHEMATICAL SCIENCES, THE
UNIVERSITY OF TOKYO, 3-8-1 KOMABA, MEGURO-KU, TOKYO, 153-8914,
JAPAN}

\textit{E-mail address}: \href{mailto:takano@ms.u-tokyo.ac.jp}{\texttt{takano@ms.u-tokyo.ac.jp}}
%%%%%%%%%%%%%%%%%%%%%%%%%%%%%%%%%%%%
\end{document}